\documentclass[11pt]{amsart}
\numberwithin{equation}{section}
\usepackage[english]{babel}
\usepackage[T1]{fontenc}
\usepackage[latin1]{inputenc}
\usepackage{indentfirst}
\usepackage{enumitem}
\usepackage{amsmath,amssymb, amsbsy}
\usepackage{comment}
\usepackage{amsfonts}
\usepackage{hyperref}
\usepackage{cleveref}
\usepackage{esint}
\usepackage{latexsym}
\usepackage{amsthm}
\usepackage[dvips]{graphicx}
\usepackage{xcolor}
\usepackage{tikz}
\usepackage{pgfplots}
\usetikzlibrary{intersections}
\usepackage{tikz-3dplot}
\usepackage[outline]{contour}
\DeclareGraphicsExtensions{.pdf,.png,.jpg,.eps}
\usepackage{amsthm}
\allowdisplaybreaks

\usepackage{bbm}

\usepackage[font=small,labelfont=bf]{caption}

\usepackage{amsthm}

\usepackage{hyperref}

\usepackage{xcolor}
\usepackage[latin1]{inputenc}
\usepackage[active]{srcltx}
\definecolor{Arancio}{cmyk}{0,0.61,0.87,0}

\definecolor{blus}{RGB}{0,102,204}

\usepackage{doi}
\newcommand{\arxiv}[1]{arXiv:\href{https://arxiv.org/abs/#1}{#1}}
\newcommand{\brd}[1]{\mathbb{#1}}
\newcommand{\R}{\brd{R}}

\newcommand{\e}{\varepsilon}

\newcommand{\ee}{\end{equation}}

\newtheorem{teo}{Theorem}[section]
\newtheorem{Corollary}[teo]{Corollary}
\newtheorem{Lemma}[teo]{Lemma}

\newtheorem{Proposition}[teo]{Proposition}
\theoremstyle{definition}
\newtheorem{Definition}[teo]{Definition}
\newtheorem{Assumption}[teo]{Assumption}

\newtheorem{remark}[teo]{Remark}

\newcommand{\supp}{\operatorname{spt}}

\newcommand{\D}{\nabla}

\newcommand{\dist}{\operatorname{dist}}
\renewcommand{\div}{\operatorname{div}}

\newcommand{\loc}{{\rm loc}}

\usepackage{amsmath}
\usepackage{tikz}
\usepackage{mathdots}
\usepackage{yhmath}
\usepackage{cancel}
\usepackage{color}
\usepackage{siunitx}
\usepackage{array}
\usepackage{multirow}
\usepackage{amssymb}
\usepackage{textcomp}
\usepackage{gensymb}
\usepackage{tabularx}
\usepackage{extarrows}
\usepackage{booktabs}
\usetikzlibrary{fadings}
\usetikzlibrary{patterns}
\usetikzlibrary{shadows.blur}
\usetikzlibrary{shapes}

\pgfplotsset{compat=1.18} 
\begin{document}

\subjclass[2020] {35B45, 35B65, 35J70}
\keywords{Higher order boundary Harnack principle; Ratio of solutions; Schauder a priori estimates; Weighted elliptic equations}

\title[Regularity Estimates for Ratio of Solutions to Elliptic Equations]{A Priori Regularity Estimates for Ratio of Solutions to Elliptic Equations with a Product Structure of Two-Dimensional Nodal Sets}

\author{Gabriele Fioravanti}

\address{Gabriele Fioravanti\newline\indent
Dipartimento di Matematica "G. Peano"
\newline\indent
Universit\`a degli Studi di Torino
\newline\indent
Via Carlo Alberto 10, 10124, Torino, Italy}
\email{gabriele.fioravanti@unito.it}

\begin{abstract}
In this paper, we establish optimal a priori $C^{1,\alpha}$ regularity estimates for the ratio $w = v/u$ of two solutions to the same elliptic equation $-\div(A \D u )=0$ with Lipschitz coefficients $A$, under the assumption that their nodal sets satisfy $Z(u) \subseteq Z(v)$. We specifically address the case where the zero set $Z(u)$ exhibits a product structure of $2$-dimensional nodal sets, namely $Z(u)=Z(u_1)\times \cdots \times Z(u_{m})$, where the $u_i$ are $2$-dimensional functions. This result extends the regularity estimates previously proved in dimension $2$ by Logunov and Malinnikova \cite{LogMal16} and by Terracini, Tortone, and Vita \cite{TerTorVit24b}.
\end{abstract}

\maketitle
\vspace{-3pt}

\section{Introduction}

The behavior of non-negative harmonic functions near the boundary is classically described by the \emph{Boundary Harnack Principle}. This result ensures that the ratio of two such functions vanishing on a common boundary portion remains bounded or H\"older continuous up to the boundary. It is a well-established fact that this property holds under weak geometric assumptions, including Lipschitz, NTA, and H\"older domains (see \cite{Kem,Dah,Anc,JerKen,B1,B2,DeSSav20}). 
When the domain exhibits higher regularity, one naturally expects the ratio to enjoy improved regularity as well. This is addressed by the \emph{Higher Order Boundary Harnack Principle}, which establishes that if the boundary is of class $C^{k,\alpha}$, the quotient inherits the same $C^{k,\alpha}$ regularity up to the boundary (see \cite{DeSSav15,BG16,TerTorVit24,Zha24,AudFioVit25}).

A more recent and challenging direction in this field involves extending the classical theory to the framework of \emph{nodal domains}. Within this setting, one examines two solutions $u$ and $v$ sharing the same uniformly elliptic equation
\begin{equation}\label{eq:elliptic}
    L_A u := -\div(A(x)\D u)  =0 \quad \text{in } B_1\subset\R^n,
\end{equation}
where the coefficient matrix $A(x)$ is symmetric and fulfills the uniform ellipticity condition
\begin{equation}\label{eq:Unif:Ell} 
   \lambda |\xi|^2 \le  A(x)\xi\cdot\xi \le \Lambda |\xi|^2 \quad \text{for a.e. } x\in B_1 \text{ and every }\xi\in\R^n,
\end{equation}
for some $0<\lambda\le\Lambda$.
The key geometric constraint here is the nodal inclusion $Z(u)\subseteq Z(v)$, where $Z(u)=\{x\in B_1 : u(x)=0\}$. Under this assumption, the primary objective in studying the quotient $v/u$ is twofold:
\begin{enumerate}
    \item[i)] to establish the effective H\"older or higher order regularity of the ratio;
    \item[ii)] to determine if this regularity is \emph{uniform} within specific classes of denominators $u$, thereby understanding the structural impact of the nodal set $Z(u)$ on the estimates.
\end{enumerate}

Before proceeding, we recall a few structural properties of the nodal set $Z(u)$. Assuming $u$ is a weak solution to \eqref{eq:elliptic} with a symmetric, Lipschitz continuous matrix $A$ satisfying \eqref{eq:Unif:Ell}, classical regularity theory yields $u\in H^{2}_{\loc}(B_1)\cap C^{1,\alpha}_{\loc}(B_1)$ for any $\alpha\in (0,1)$. Furthermore, following \cite{Lin91,Han94}, we can decompose the nodal set as $Z(u)=R(u)\cup S(u)$, where the regular and singular parts are respectively given by
\[
R(u)=\{x\in Z(u):|\D u(x)|\neq 0\} \quad \text{and} \quad S(u)=\{x\in Z(u):|\D u(x)|=0\}.
\]
An immediate consequence of the Implicit Function Theorem is that $R(u)$ is locally $C^{1,\alpha}$ for every $\alpha\in(0,1)$. On the other hand, the singular set $S(u)$ has Hausdorff dimension bounded by $n-2$, reducing to a locally finite set in dimension $n=2$.

Outside the singular set $S(u)$, the analysis of the ratio $v/u$ under the inclusion $Z(u) \subseteq Z(v)$ reduces to the classical Higher Order Boundary Harnack Principle, which guarantees local $C^{1,\alpha}$ regularity for the quotient (see \cite{DeSSav15,TerTorVit24}).
In contrast, near the singular set $S(u)$, one cannot rely on these classical techniques, as the boundaries of the \emph{nodal domains} -- the connected components of $\{x\in B_1:u(x) \neq 0\}$ -- exhibit very poor regularity, being merely Lipschitz or even worse. Nevertheless, unlike in the classical Boundary Harnack Principle where the domain is imposed a priori, here the  "boundary" naturally arises as the zero set of a solution to an elliptic equation. Because of this intrinsic PDE structure, one might actually expect the ratio $v/u$ to enjoy higher regularity than what standard theory ensures for such rough domains.

The pioneering results in this direction are due to Logunov and Malinnikova \cite{LogMal15,LogMal16}. Specifically, under the assumption of analytic coefficients, they showed that the ratio $v/u$ is analytic across $Z(u)$. Furthermore, in dimension $n=2$, they proved uniform $C^{k,\alpha}$ estimates for the ratio within the class of denominators $u$ having an upper bound on the number of nodal domains.

Lowering the regularity framework, Lin and Lin \cite{LinLin22} addressed the problem under Lipschitz assumptions on the coefficients $A$. In every dimension, they established H\"older continuity for $v/u$ across the entire nodal set $Z(u)$ for some implicit exponent $\alpha^* \in (0,1)$, showing that such regularity estimates are uniform within the class of denominators $u$ having an upper bound on the generalized Almgren frequency functional.

More recently, in a series of works \cite{TerTorVit24,TerTorVit25,TerTorVit24b}, Terracini, Tortone, and Vita addressed both the effective regularity and the uniform estimates under Lipschitz assumptions on the coefficients. First, in every dimension and for every $\alpha\in(0,1)$, they obtained \emph{a priori} $C^{0,\alpha}$ estimates which are uniform within the class of denominators $u$ with a bounded generalized Almgren frequency functional. In dimension $n=2$, they proved that this class is equivalent to the one considered by Logunov and Malinnikova. Moreover, in the planar case, they established optimal \emph{a priori} $C^{1,\alpha}$ estimates for every $\alpha\in(0,1)$, which are uniform within the same class of denominators. Next, by means of a regularization-approximation scheme, they were able to prove the actual $C^{1,\alpha}$ regularity of the ratio.

The aim of this paper is to extend the $2$-dimensional uniform \emph{a priori} $C^{1,\alpha}$ estimates established by Logunov and Malinnikova and by Terracini, Tortone, and Vita to a higher-dimensional framework and in the case of Lipschitz continuous coefficients. Here, we address the case where the zero set $Z(u)$ exhibits a product structure of $2$-dimensional nodal sets, namely $Z(u)=Z(u_1)\times \cdots \times Z(u_{m})$, where the $u_i$ are $2$-dimensional functions (as detailed in Assumption \eqref{A:orthogonal} below). Under this geometric assumption, we establish optimal \emph{a priori} $C^{1,\alpha}$ estimates which are uniform within this frequency bounded class of denominators.

\subsection{Main Result}

To state our main result, we first introduce the necessary definitions and define the uniformity classes of coefficients and denominators.

Given $0<\lambda\le \Lambda $ and $ L\ge0$, we define the class of admissible coefficients $\mathcal{A}:=\mathcal{A}_{\lambda,\Lambda,L}$ as 
\[
\mathcal{A}:=\Big\{ A\in C^{0,1}(B_1,\R^{n,n}) : A=A^T \text{ satisfies }\eqref{eq:Unif:Ell} \text{ and } [A]_{C^{0,1}(B_1)}\le L \Big\}.
\]
Given $N_0>0$, we define the class $\mathcal{S}_{N_0}:=\mathcal{S}_{N_0,\lambda,\Lambda,L} $ as the solutions to \eqref{eq:elliptic}, for which the generalized Almgren frequency formula $N$ satisfies an upper bound, that is, 
\[
\mathcal{S}_{N_0}:=\Big\{ u\in H^1(B_1) : u \text{ solves } \eqref{eq:elliptic} \text{ with } A\in\mathcal{A}, \, 0 \in Z(u),  \, N(0,u,1) \le N_0, \, \|u\|_{L^2(B_1)} =1\Big\}.
\]
This class is the natural choice for seeking uniform estimates, as it ensures the compactness of the denominator in the blow-up analysis, as established in \cite{TerTorVit25}.

In our main result we make the following orthogonal assumption about the coefficients and the denominators $u$.

\begin{Assumption}[Orthogonality assumption]\label{A:orthogonal}
    Let us suppose that $n$ is an even number.
    Let us consider $A\in\mathcal{A}$. We assume that $A$ can be written in block form as 
    \[
    A(x) := \begin{pmatrix}
       A_1(x_1,x_2) & 0 & \cdots & 0 \\ 
       0 & A_2(x_3,x_4) & \cdots & 0 \\
       \vdots & \vdots & \ddots & \vdots \\
       0 &  0 & \cdots & A_{n/2}(x_{n-1},x_n) 
    \end{pmatrix},
    \]
    where $A_i$ is a $2$-dimensional matrix and every $A_i$ depends on the couple of variables $(x_{2i-1},x_{2i})$.
    We also assume that $u\in \mathcal{S}_{N_0}$ can be written in the form
    \[
u:=\prod_{i=1}^{n/2} u_i(x_{2i-1},x_{2i}),
\]
and every $u_i$ satisfies $L_A u_i= 0$ in $B_1$.

Instead, if $n$ is an odd number we consider the last block $A_{(n+1)/2}$ as a real-valued function and not a $2$-dimensional matrix and the last function $u_{(n+1)/2}$ depends only on $x_n$.

\end{Assumption}

Under the previous assumption, we prove optimal, uniform-in-$\mathcal{S}_{N_0}$ \emph{a priori} $C^{1,\alpha}$ estimates for the ratio $v/u$ of solutions to \eqref{eq:elliptic} satisfying $Z(u)\subseteq Z(v)$. For the sake of brevity, we say that a constant depends on the class $\mathcal{S}_{N_0}$ if it depends on $n$, $\lambda$, $\Lambda$, $L$, and $N_0$.

\begin{teo}[A priori $C^{1,\alpha}$ estimates]\label{T:1}
Let $\alpha\in(0,1)$, $A\in\mathcal{A}$, $u\in \mathcal{S}_{N_0}$ satisfying the Assumption \ref{A:orthogonal} and $v$ be a solution to
\[
L_A u=L_Av=0\quad  \text{in }B_1, \qquad \text{such that }{Z}(u)\subseteq {Z}(v).
\]
Then, if $v/u \in C^{1,\alpha}(B_1)$, there exists a constant $C>0$, which depends only on the class $\mathcal{S}_{N_0}$ and $\alpha$, such that
\begin{equation}\label{eq:T:1}
{\Big\| \frac{v}{u}  \Big\|_{C^{1,\alpha}(B_{1/2})}}\le C \|v\|_{L^2(B_1)}.
\end{equation}
\end{teo}

The strategy of the proof highlights a deep connection with the theory of degenerate elliptic PDEs. Indeed, as proved in \cite{TerTorVit24,TerTorVit25,TerTorVit24b}, the ratio $w:=v/u$ satisfies (in the sense of Definition \ref{D:sol}) the weighted elliptic equation
\[
-\operatorname{div}(u^2 A \nabla w) = 0 \quad \text{in } B_1,
\]
where the coefficient matrix $u^2 A$ degenerates on the nodal set $Z(u)$.

Such equations are highly degenerate, and there are few results in the literature concerning this framework; in general, solutions are not even continuous across $Z(u)$ (see Example 1.4 in \cite{STV21}). Despite this, by adapting the classical blow-up argument due to Simon \cite{Sim97}, as done in \cite{TerTorVit24b}, we are able to obtain Schauder $C^{1,\alpha}$ estimates for the ratio $w$. 
The proof of these estimates crucially requires controlling the oscillation of $\D w$ near the singular set $S(u)$ in a quantitative way (see Section \ref{S:4.1}). Under the structural Assumption \ref{A:orthogonal}, we extend the $2$-dimensional strategy of \cite{TerTorVit24b} to higher dimensions, but this transition introduces a major structural difference.
Specifically, by finding an appropriate way to project $\nabla w$ onto each $R(u_i)$ -- the planar regular set of the factor $u_i$ -- and by exploiting the conormal boundary condition satisfied by $w$ on the regular set $R(u)$, combined with a fine blow-up analysis, we ensure a uniform bound on the oscillation of $\nabla w$. This allows us to track and control the behavior of the solution on every $2$-dimensional component separately (see also \cite{CoFiPaVi25} for a similar framework concerning elliptic equations with monomial weights). We emphasize that this component-wise approach is a new feature of the higher-dimensional setting, as it has no counterpart in the $2$-dimensional strategy of \cite{TerTorVit24b}, and it could represent a first step toward tackling more general, non-product geometric configurations of the zero set $Z(u)$.

\section{Preliminaries}

In this section, we begin by recalling some basic facts about solutions to \eqref{eq:elliptic} with Lipschitz coefficients, focusing on the associated frequency function which plays a central role in our local analysis. 

As a starting point, let us consider the case of constant coefficients (assuming $A=\mathbb{I}$ for simplicity). For $x_0\in B_1$ and $r<1-|x_0|$, we define
\[
H(x_0,u,r):=\frac{1}{r^{n-1}}\int_{\partial B_r} u^2 \, d\sigma,\qquad D(x_0,u,r):=\frac{1}{r^{n-2}} \int_{B_r} |\D u|^2\, dx,
\]
and the frequency function as
\[
N(x_0,u,r):=\frac{D(x_0,u,r)}{H(x_0,u,r)}.
\]
As first observed by Almgren, the map $r \mapsto N(x_0,u,r)$ is monotone non-decreasing.

This theory was successfully extended to the case of variable Lipschitz coefficients by Garofalo and Lin \cite{GarLin86}. They introduced a \emph{generalized} Almgren frequency $N(x_0,u,r)$ depending explicitly on $A$, proving that the Lipschitz continuity of the matrix ensures an \emph{almost} monotonicity property for this functional. This monotonicity property serves as a powerful tool to control the vanishing order of solutions, thereby implying the strong unique continuation principle. At the same time, this tool provides quantitative bounds on the size of both the critical and nodal sets. We skip the construction of \emph{generalized} Almgren frequency functional and we refer to \cite{GarLin86, Lin91, Han94, CheNabVal15, NabVal17, TerTorVit25, TerTorVit24b} for the exact definition and related results. The almost monotonicity of the generalized frequency is stated in the following proposition.

\begin{Proposition}[\cite{GarLin86,CheNabVal15,TerTorVit25}]
    Let $A \in \mathcal{A}$, $u$ be a solution to \eqref{eq:elliptic}, $x_0\in B_1$. Let $N(x_0,u,x)$ be the generalized Almgren frequency as defined in \cite[\S 2]{TerTorVit25}.   
    Then, there exists a constant $C>0$ depending on $n$, $\lambda$, $\Lambda$, $L$, such that
    \[
    r \mapsto  e^{Cr} N(x_0,u,r),
    \]
    is monotone non-decreasing for $r< \Lambda^{-1/2}(1-|x_0|)$.
\end{Proposition}

Among the consequences of the previous proposition, there is the following doubling-type estimate, which plays a crucial role in the rest of the paper.

\begin{Lemma}\label{L:doubling}
Let $ A\in\mathcal{A}$ and let $u\in H^1(B_1)$ be a weak solution to $L_A u = 0 $ in $B_1$. For every $x_0 \in B_{1/2}$, $0<r\le R \le \Lambda^{-1/2}(1-|x_0|)$, there exists a constant $C>0$, depending on $n$, $\lambda$, $\Lambda$, $L$, such that
  \begin{equation*}
      \fint_{B_{R}(x_0)}u^2 \, dx
      :=\frac{1}{|B_R|}\int_{B_R(x_0)} u^2\, dx
      \le C \Big(\frac{R}{r}\Big)^{2N(x_0,u,R)}\fint_{B_{r}(x_0)}u^2 \, dx,
  \end{equation*}
   and 
    \begin{equation*}
      \fint_{\partial B_{R}(x_0)}u^2 \, d\sigma :=\frac{1}{|\partial B_R|}\int_{\partial B_R(x_0)} u^2\, d\sigma \le C \Big(\frac{R}{r}\Big)^{2N(x_0,u,R)}\fint_{\partial B_{r}(x_0)}u^2 \, d\sigma.
  \end{equation*}
\end{Lemma}

\subsection{Weak solution}

Next, we briefly recall the notion of weak solution to
\begin{equation}\label{eq:u^2:solution}
    -\div(u^2 A \D w) = 0 \quad \text{in } B_1, \qquad \text{with } u \in \mathcal{S}_{N_0},
\end{equation}
as introduced in \cite[\S 3.2]{TerTorVit25}. For a more general discussion of the framework we refer the reader to the aforementioned reference.

Given $u\in\mathcal{S}_{N_0}$, we define the weighted Sobolev space $H^{1}(B_1,u^2\, dx)$ as the completion of $C^\infty(\overline{B_1})$ with respect to the norm
\[
\|w\|_{H^{1}(B_1,u^2 \, dx)}:=\Big(
\int_{B_1}u^2\big(
w^2 + |\D w|^2
\big) \, dx
\Big)^{1/2}.
\]
We then introduce the class of solutions which we focus on in the proof of the main result.
\begin{Definition}\label{D:sol}
    Let $A\in \mathcal{A}$ and $u \in \mathcal{S}_{N_0}$. We say that $w$ is a weak solution to \eqref{eq:u^2:solution},
    if $w\in H^{1}(B_1,u^2\, dx)$ and 
    \[
    \int_{B_1}u^2 A\D w\cdot \D \phi \, dx= 0,\quad \text{for every } \phi \in C_c^\infty(B_1).
    \]
\end{Definition}

\begin{remark}\label{R:ratio}
    As proved in \cite[Proposition 3.5]{TerTorVit24}, the ratio of two solutions $v,u$ to \eqref{eq:elliptic} satisfies \eqref{eq:u^2:solution} in the sense of the Definition \ref{D:sol}. 
    Moreover, in \cite[Proposition 3.9]{TerTorVit24b}, the authors established an inverse relation for this ratio, which can be stated as follows. Let $u\in\mathcal{S}_{N_0}$ and $w \in H^1(B_1,u^2\, dx)$ be a weak solution to \eqref{eq:u^2:solution}. Assume that $w\in C(B_1\setminus S(u))$. Then, the function $v:=wu \in H^1(B_1)$ is a weak solution to \eqref{eq:elliptic} and $Z(u)\subseteq Z(v)$. 

\end{remark}

\section{Harmonic polynomials and invariance directions}

In this section, we present several results concerning harmonic polynomials and solutions to elliptic equations. We emphasize that these results hold generally and do not rely on Assumption \ref{A:orthogonal}.
The following result shows that if the zero set of a harmonic polynomial is tangential to a given direction, then the polynomial is invariant in that direction.

\begin{Lemma}\label{L:dim:red:pol}
    Let $A$ be a constant symmetric positive matrix and $P:\R^n\to\R$ be a polynomial such that $L_A P = 0$ in $\R^n$. Let us suppose that $\partial_{\bf e}P=\D P\cdot {\bf e} =0$ for some unit vector ${\bf e} \in \mathbb{S}^{n-1} $ on the regular set $Z(P):=\{P=0\}$. Then, $\partial_{\bf e}P =0$ on the whole $\R^n$.
\end{Lemma}

\begin{proof}
    Without loss of generality, we may assume $A = \mathbb{I} $ and $\mathbf{e} = e_n $. In the general case, it is sufficient to consider the transformed polynomial $\tilde{P}(x) := P(A^{1/2}x) $, which satisfies $\Delta \tilde{P} = 0 $, and then apply a suitable orthogonal transform.

    Since $\partial_{x_n} P = 0 $ on $Z(P)$, it follows that $Z(P) \subseteq Z(\partial_{x_n} P)$. By Murdoch's division lemma (see \cite[Theorem 2, Lemma 4]{Mur64} and \cite[Lemma 5.6]{TerTorVit25}), there exists a polynomial $Q$ such that $\partial_{x_n}P/P = Q$. Since the degree of $\partial_{x_n}P$ is less or equal than the degree of $P$, this forces $Q=0$, and consequently $\partial_{x_n}P = 0$.
\end{proof}

\begin{Corollary}\label{Cor:linear:poly}
      Let $A$ be a constant symmetric positive matrix and $P:\R^n\to\R$ be a polynomial such that $L_A P = 0$ in $\R^n$ and fix $\e\in(0,1)$. Then, the following are equivalent:
     \begin{itemize}
     \item[i)] $P$ is a linear function,  
     \item [ii)]  $\D P(x)=|\D P(x)|{\bf e}$, for every $x\in{R}(P),$ where ${\bf e}\in {\mathbb{S}}^{n-1}$,
     \item [iii)] $N_\infty:=\lim_{r\to\infty}N(0,P,r)\le 1+\e$.
     \end{itemize}
\end{Corollary}

\begin{proof}
The equivalence between $i)$ and $ii)$ immediately follows by Lemma \ref{L:dim:red:pol}. If $iii)$ holds true, by standard blow-down argument (see \cite[Proposition 5.1]{TerTorVit24b} and \cite[Corollary 4.7]{LinLin22}) we get that $P$ would be a polynomial of degree $N_\infty\le 1+\e$, so $N_\infty=1$, $P$ is linear and $i)$ holds true. On the other hand, $i)$ implies $iii)$ follows trough a straightforward computation.
\end{proof}

\begin{Corollary}\label{Cor:quadratic:pol}
          Let $A$ be a constant symmetric positive matrix and $P:\R^n\to\R$ be a polynomial such that $L_A P = 0$ in $\R^n$ and
         \[
         \lim_{r\to\infty}N(0,P,r)\le \overline N,
         \]
         for some constant $\overline N>0$. Then, the following are equivalent:
        
     \begin{itemize}    
     \item[i)] the degree of $P$ is greater than one,  
     \item [ii)] there exist $R>0$, $\delta \in (0,1)$, and two points $x_1,x_2 \in {R}(P)\cap B_R$ such that 
     \[
    \Big|\frac{\D P(x_1)}{|\D P(x_1)|}\cdot \frac{\D P(x_2)}{|\D P(x_2)|} \Big|\le 1-\delta,
     \]
     \item [iii)] there exist $\e>0$ and $\rho>0$ such that     $N(0,P,\rho)\ge 1+\e$.
     \end{itemize}
\end{Corollary}

\begin{proof}
Without loss of generality take $A=\mathbb{I}$.
    The equivalence between $i)$ and $iii)$  is the same as the one established between $i)$ and $iii)$ in Corollary \ref{Cor:linear:poly}.  Next, by contradiction, let us assume that $ii)$ holds true and $i)$ does not hold, that is, the degree of $P$ is less than one. Then, for all points in $x\in \R^n$ one has that $\D P(x)$ is constant. In particular, since the polynomial $P$ is not constant, $\D P(x)$ is constant on ${R}(P)$ and this contradicts the assumption $ii)$. Finally, by contradiction, let us assume that $iii)$ holds true and that $ii)$ is false. So, there exists a sequence of harmonic polynomials $P_k$ such that
    \[
    N(0,P_k,\rho)\ge 1+\e, \quad \lim_{r\to\infty}N(0,P_k,r)\le \overline N,\quad 
     \Big|\frac{\D P_k(x_1)}{|\D P_k(x_1)|}\cdot \frac{\D P_k(x_2)}{|\D P_k(x_2)|} \Big|\geq 1-k^{-1},
    \]
    for every $x_1,x_2 \in {R}(P_k)\cap B_{k}$. Considering the normalized sequence ${P_k}/{\|P_k\|_{L^2(\partial B_1)}}$ and using standard compactness argument via Schauder estimates, we obtain that $P_k \to P_\infty$ in $C^{1,\alpha}_\loc(\R^n)$ where $P_\infty$ is a polynomial with degree $N_\infty$, which satisfies
    \[
    \Delta P_\infty =0,\quad  1+\e\le N_\infty\le \overline{N},\quad \|P_\infty\|_{L^2(\partial B_1)}=1,\quad \frac{\D P_\infty(x)}{|\D P_\infty(x)|}={\bf e}\in\mathbb{S}^{n-1}, \ \forall x\in {R}(P_\infty).
    \]
    Since $N_\infty\ge 1+\e>1$, one has that $P_\infty$ has degree greater than one, which is a contradiction with the fact that $\D P_\infty(x)=|\D P_\infty(x)|{\bf e}$, for every $x\in{R}(P_\infty)$, see Corollary \ref{Cor:linear:poly}.
\end{proof}

As consequence of the previous results, we obtain that following lemma,  which extends \cite[Lemma 2.9]{TerTorVit24b} to the $n$-dimensional setting.

\begin{Lemma}\label{L:2:points}
Fix $\e>0$ and $\overline{N}>0$. There exist small $\delta>0$ and large $R>1$ such that the following holds true. Let us suppose that $A\in\mathcal{A}$ and $U\in H^1(B_R)$ such that
\[
[A]_{C^{0,1}(B_R)}\le\delta, \qquad L_A U = 0 \text{ in }B_R,\qquad 0\in {R}(U) \qquad N(0,U,R)\le \overline{N},
\]
and
\[
N(0,U,1)\ge 1+\e.
\]
Then, there exist $R'\in(0,R)$, $\e'\in(0,1)$ depending only on $\overline{N}$, $\e$, and there exists a point $x_1 \in {R}(U)\cap B_{R'}$ such that
    \[
        \Big|\frac{\D U(x_1)}{|\D U(x_1)|}\cdot \frac{\D U(0)}{|\D U(0)|}\Big| \le 1 - \e'.
    \]
\end{Lemma}

Before proceeding to the proof, we recall the following compactness result.

\begin{remark}\label{R:normalized:convergence}  Let us consider a sequence of matrices $A_k \in \mathcal{A}$ and a sequence of non trivial functions $u_k\in H^1(B_k)$ such that
    \[
    [A_k]_{C^{0,1}(B_k)}\le k^{-1}, \qquad L_{A_k} u_k = 0 \text{ in }B_k, \qquad N(0,u_k,k)\le \overline{N}.
    \]
    Then, $A_k \to \bar A$ uniformly on every compact set $K\subset \R^n$, where $\bar A$ is a constant uniformly elliptic matrix and the normalized sequence $\tilde u_k := u_k/\|u_k\|_{L^2(\partial B_1)}$ satisfies 
    \[
    L_{A_k} \tilde u_k = 0 \text{ in }B_k, \qquad N(0,\tilde u_k,k)\le \overline{N},\qquad \|\tilde u_k\|_{L^2(\partial B_1)}=1,
    \]
    For a fixed $\bar{R}>1$, applying classical Schauder estimates yields that $\tilde u_k$ is uniformly bounded in $C^{1,\alpha}(B_{\bar R})$ and by employing a diagonal argument we get $\tilde u_k \to P$ in $C^{1,\alpha}_\loc(\R^n)$ and $P$ satisfies
    \[
    L_{\bar A}P=0 \text{ in } \R^n,\qquad \|P\|_{L^2(\partial B_1)}=1,\qquad N_\infty:=\lim_{t\to\infty}N(0,P,t)\le \overline{N}.
    \]
    By a blow-down argument (see \cite[Proposition 5.1]{TerTorVit24b} and \cite[Corollary 4.7]{LinLin22}), we get that $P$ is a non trivial polynomial of degree $N_\infty\le\overline N$.
\end{remark}

\begin{proof}[Proof of Lemma \ref{L:2:points}]    
Arguing by contradiction, let us suppose that there exist sequences $A_k\in \mathcal{A}$, $U_k \in H^1(B_k)$ satisfying
    \[
    [A_k]_{C^{0,1}(B_k)}\le k^{-1}, \quad L_{A_k} U_k = 0 \text{ in }B_k, \quad 0 \in {R}(U_k),\quad N(0,U_k,k)\le \overline{N}, \quad N(0,U_k,1)\ge 1+\e,
    \]
    and
    \[
    \Big|\frac{\D U_k(x_1)}{|\D U_k(x_1)|}\cdot \frac{\D U_k(0)}{|\D U_k(0)|} \Big|\ge 1- k^{-1},\quad\text{for every }x_1 \in {R}(U_k).
    \]
    Considering $\tilde U_k:={U_k}/{\|U_k\|_{L^2(\partial B_1)}}$ and using Remark \ref{R:normalized:convergence}, we get that
     $\tilde U_k \to P$ in $C^{1,\alpha}_\loc(\R^n)$ and $P$ is a polynomial satisfies
    \[
    L_{\bar{A}} P = 0 \ \text{ in }\R^n, \quad \|P\|_{L^2(\partial B_1)}=1, \quad  1+ \e\le  \lim_{t\to\infty}N(0,P,t)\le \overline{N},
    \]
    where $\bar{A}$ is a constant uniformly elliptic matrix. 
 Therefore, $P$ has degree greater or equal than two since $\e>0$ (see Corollary \ref{Cor:quadratic:pol}). On the other hand, we have
    \[
    \frac{\D P(x_1)}{|\D P(x_1)|}\cdot \frac{\D P(0)}{|\D P(0)|} = 1,\quad\text{for every }x_1 \in {R}(P),
    \]
    which means that $\D P(x)=|\D P(x)|{\bf e}$, for every $x_1\in{R}(P)$, where ${\bf e}\in\mathbb{S}^{n-1}$. This is a contradiction with Corollary \ref{Cor:linear:poly}.
\end{proof}

\section{A priori Schauder estimates}

This section is devoted to the proof of Theorem \ref{T:1}. Before addressing the proof, we start by proving a simple lemma concerning the control of the $L^\infty$-norm of the gradient of solutions to \eqref{eq:elliptic} away from the zero set of $u$.

\begin{Lemma}\label{L:Linfty:grad}
     Let $\overline{N}>0$. Let us suppose that $A\in\mathcal{A}$, $u\in H^1(B_1)$ such that
\[
\quad L_A u = 0 \text{ in }B_1, \quad N(0,u,1)\le \overline{N}.
\]
For every $x_0 \in B_1\setminus {Z}(u)$, set $\delta :=\dist(x_0, {Z}(u))$. Then, there exists $C>0$, which not depends on $\delta$ such that
\begin{equation}\label{eq:very:regular:bad:case2}
       \|\D u\|_{L^\infty(B_{\delta}(x_0))}\le \frac{C}{\delta}|u(x_0)|.
\end{equation}
\end{Lemma}

\begin{proof}
By using the scaling invariant Schauder estimates to $u$ in $B_\delta(x_0)$, the doubling property (see Lemma \ref{L:doubling}) and the Harnack inequality in $B_{\delta/2}(x_0)$, we get
\[
\|\D u\|^2_{L^{\infty}(B_\delta(x_0))} \le \frac{C}{\delta} \fint_{B_{2\delta}(x_0)}u^2\,dx \le  \frac{C}{\delta} \fint_{B_{\delta/2}(x_0)} u^2 \,dx\le  \frac{C}{\delta}\inf_{B_{\delta/2}(x_0)}|u|^2 \le  \frac{C}{\delta} |u(x_0)|^2. \qedhere
\]
\end{proof}

Next, we recall a Liouville-type theorem for the ratio of entire solutions.

\begin{teo}[{\cite[Theorem 1.2]{TerTorVit25}}]\label{T:liouville}
    Let $A$ be a constant uniformly elliptic matrix, $u$ be a polynomial and $v$ such that 
    \[-\div(A\D u) = -\div(A \D v)=0  \quad\text{in }\R^n,\]
    and $Z(u) \subseteq Z(v)$. Suppose that there exist $\gamma\ge0$, $C>0$ such that \[
    \Big|\frac{v}{u}(x)\Big|\le C(1+|x|)^\gamma, \quad\text{in }\R^n.
    \]
Then, the ratio $v/u$ is a polynomial of degree at most $\lfloor \gamma\rfloor$.
\end{teo}

We are finally ready to provide the proof of our main result.

\begin{proof}[Proof of Theorem \ref{T:1}] 
Let us suppose that $n$ is an even number (the odd case follows analogously) and call $w:=v/u$. In light of \cite[Proposition 3.11]{TerTorVit25}, which ensures the local boundedness of the ratio $w$, it is sufficient to prove that
\[
{\| w  \|_{C^{1,\alpha}(B_{1/2})}}\le C \|w\|_{L^\infty(B_1)},
\]
for some $C>0$ depending only on the class $\mathcal{S}_{N_0}$.
Equivalently, following the strategy of \cite[Theorem 2.28]{FerRos22}, we aim to show that for every sufficiently small $\delta>0$ there exists $C_\delta>0$ such that
\begin{equation}\label{eq:1:alpha:proof}
      [\D w]_{C^{0,\alpha}(B_{1/2})}\le \delta
 [\D w]_{C^{0,\alpha}(B_{1})} + C_\delta \big( \|\D w\|_{L^\infty(B_{1})} + \| w\|_{L^\infty(B_{1})}\big).
\end{equation}
Using the sub-additivity of the H\"older seminorms with respect to unions of convex sets (see \cite[Lemma 2.27]{FerRos22}) and a covering argument, it follows that \eqref{eq:1:alpha:proof} implies \eqref{eq:T:1}.

In order to establish \eqref{eq:1:alpha:proof}, we argue by contradiction. Let us assume that there exist sequences $A_k \in \mathcal{A}$, $u_k \in \mathcal{S}_{N_0}$, and $v_k$ be solutions to 
\[
L_{A_k} u_k=L_{A_k} v_k =0 \quad \text{in }B_1, \quad \text{such that }{Z}(u_k)\subseteq {Z}(v_k),
\]
$w_k:=v_k/u_k \in C^{1,\alpha}(B_1)$ and there exists $\delta_0>0$ such that
\begin{equation}\label{eq:1:alpha:contradiction}
    [\D w_k]_{C^{0,\alpha}(B_{1/2})}>\delta_0
 [\D w_k]_{C^{0,\alpha}(B_{1})} + k \big( \|\D w_k\|_{L^\infty(B_{1})} + \| w_k\|_{L^\infty(B_{1})}\big).
 \end{equation}
By definition of H\"older seminorm, let us consider two sequences of points $x_k,y_k \in B_{1/2}$ such that 
\begin{equation}\label{eq:two:points}
    \frac{|\D w_k(x_k)-\D w_k(y_k)|}{|x_k-y_k|^\alpha} \ge \frac{[\D w_k]_{C^{0,\alpha}(B_{1/2})}}{2},
\end{equation}
and set $r_k:=|x_k-y_k|$. By using \eqref{eq:1:alpha:contradiction}, it readily follows that $r_k\to 0$, in fact
\[
r_k^\alpha \le \frac{|\D w_k(x_k)-\D w_k(y_k)|}{2 [\D w_k]_{C^{0,\alpha}(B_{1/2})}} \le \frac{\| \D w_k\|_{L^\infty(B_{1})}}{[\D w_k]_{C^{0,\alpha}(B_{1/2})}} \le k^{-1}\to 0.
\]
Let us write $x_{k}:=((x_{k})_1,\dots,(x_{k})_n) \in \R^n$ and for every $i=1,\dots,n/2$ let us define the sequence of points 
\[
x_{i,k}:=((x_{k})_{2i-1},(x_{k})_{2i}) \in\R^2.
\] 
We denote by $\xi_{i,k}\in\R^2$ a chosen projection of $x_{i,k}$ on the  two-dimensional regular set ${R}(u_{i,k})$ restricted to $\R^2$ and define $\delta_{k}^i:=|x_{i,k}-\xi_{i,k}|$. We now define 
\begin{equation}\label{eq:hat:x:ik}
\hat{x}_{i,k}:=\begin{cases}
    x_{i,k}, & \text {if } \delta_{i,k}/r_k \to \infty,\\
    \xi_{i,k},  & \text {if } \delta_{i,k}/r_k \le C,
\end{cases}
\end{equation}
and $\hat x_k:=(\hat x_{1,k},\dots,\hat x_{n/2,k}) \in B_{3/4}$.

Next, define the sequence of functions
\[
W_k(x):=\frac{w_k(\hat{x}_k+r_k x) - w_k(\hat x_k) - {\D w_k(\hat x_k)} \cdot r_k x}{r_k^{1+\alpha}[\D w_k]_{C^{0,\alpha}(B_{1})}}, \quad \text{for } x\in \Omega_k:= \frac{B_1-\hat{x}_k}{r_k}.
\]
For every compact set $K\subset \R^n$, it is standard to verify that
\[
[\D W_k]_{C^{0,\alpha}(K)} \le 1.
\]
Moreover, since $ W_k(0) = 0 $ and $ |\D W_k(0)| = 0 $, we can apply the first-order expansion of $ W_k $ at the origin to obtain a local bound on its $ C^{1,\alpha} $-norm. In particular, for every compact set $ K \subset \Omega_k $, we deduce that
\[
\| W_k\|_{C^{1,\alpha}(K)} \le C_K,
\]
for some constant $ C_K>0 $ depending only on the diameter of $ K $. By a standard compactness argument via Arzel\`a-Ascoli theorem, we then obtain that $W_k\to W$ in $C^{1}_\loc(\R^n)$ and $W$ satisfies the growth condition
\begin{equation}\label{eq:growth:W}
    |W(x)|\le C(1+|x|)^{1+\alpha}, \quad\text{ for every } x\in\R^n.
\end{equation}
At this point, let us consider the sequences of points
\[
    P_k := \frac{x_k - \hat x_k}{r_k}, \quad \hat P_k := \frac{y_k - \hat x_k}{r_k}.
\]
Using \eqref{eq:two:points} and \eqref{eq:1:alpha:contradiction}, we deduce that
\[
|\D W_k(P_k) - \D W_k(\hat P_k)| = \frac{|\D w_k(x_k) - \D w_k(y_k)|}{r_k^\alpha [\D w_k]_{C^{0,\alpha}(B_1)}}
\ge \frac{1}{2}  \frac{[\D w_k]_{C^{0,\alpha}(B_{1/2})}}{[\D w_k]_{C^{0,\alpha}(B_1)}}
\ge \frac{\delta_0}{2}.
\]
Moreover, by the choice of $\hat x_k$ it follows that $ P_k $ and $ \hat P_k $ lie inside a compact set. Therefore, one has that $P_k\to P$, $\hat{P}_k\to \hat P$ and $P\not=\hat P$. Then, by the uniform convergence $ \D W_k \to \D W $ it follows that $ \D W $ is not constant.

The next step is to look at the equation satisfied by $W_k$, recalling that $w_k=v_k/u_k$ is a weak solution to
\[
-\div(u_k^2 A_k \D w_k)=0 \quad \text{in } B_1.
\]
and, by \cite[Theorem 1.2]{TerTorVit24}, satisfies
\begin{equation}\label{eq:conormal:wk}
    \D u_k \cdot A_k\D w_k = 0  \quad \text{on } {R}(u_k).
\end{equation}
Setting $\bar A_k:=A_k(\hat x_k+r_k x)$ and using standard computation, it follows that, for every $\phi\in C_c^\infty(\R^n)$ and $k$ sufficiently large 
\begin{equation}\label{eq:W_k:not:normalized}
\int u_k^2(\hat x_k+r_k x) \bar A_k(x) \D W_k(x) \cdot \D \phi(x) \,dx = - \frac{1}{ r_k^{\alpha} [\D w_k]_{C^{0,\alpha}(B_1)} } \int u_k^2(\hat x_k+r_k x) \bar A_k(x) {\D w_k(\hat x_k)} \cdot \D \phi(x) \,dx.
\end{equation}
Let us recall that $u_k(x):=\prod_{i=1}^{n/2} u_{i,k}(x_{2i-1,2i})$ and define the rescaled functions
\[
U_{i,k}(x):=\frac{u_{i,k}(\hat x_k+r_k x)}{H(\hat x_k,u_{i,k},r_k)^{1/2}}, \qquad H(\hat x_k,u_{i,k},r_k):=\fint_{\partial B_{r_k}}u_{i,k}^2 \, d\sigma, \qquad  U_k(x):= \prod_{i=1}^{n/2} U_{i,k}(x).
\]
Dividing both sides of \eqref{eq:W_k:not:normalized} by 
$H(\hat x_k,u_k,r_k)=\prod_{i=1}^{n/2}H(\hat x_k,u_{i,k},r_k)$, we obtain
\begin{equation}\label{eq:Wk:normalized}
    \int U_k^2 \bar A_k \D W_k \cdot \D \phi\,dx = -\frac{1}{ r_k^{\alpha} [\D w_k]_{C^{0,\alpha}(B_1)} }\int U_k^2 \bar A_k {\D w_k(\hat x_k)} \cdot \D \phi \,dx.
\end{equation}
We now aim to pass to the limit as $ k \to \infty $ in the previous identity in order to derive a contradiction. We begin by analyzing the left-hand side. First, recalling that $ {\|A_k\|}_{C^{0,1}(B_1)} \le L $, it follows that the rescaled coefficients satisfy $ {[\bar{A}_k]}_{C^{0,1}(K)} \le L r_k $, for every compact set $K\subset \R^n$. Therefore, by the Arzel\`a-Ascoli theorem, we have $ \bar{A}_k \to \bar{A} $ uniformly on every compact sets of $\R^n$, where $ \bar{A} $ is a constant uniformly elliptic matrix. By using the compactness of the sequences $U_{i,k}$ (see \cite[Proposition 2.5]{TerTorVit25}), we get that $U_{i,k}\to U_i$ uniformly on every compact sets of $\R^n$, where $U_i$ is a global polynomial solution to $-\div(\bar A \D U_i)=0$ in $\R^n$. Subsequently, by Assumption \ref{A:orthogonal}, the product function $U = \prod_{i=1}^{n/2} U_i$ is also a global polynomial solution to  $-\div(\bar A \D U)=0$ in $\R^n$. Finally, we have shown before that $W_k \to W$ in $C^{1}_\loc(\R^n)$. Hence, combining these results it follows that 
\[
\int U_k^2 \bar A_k \D W_k \cdot \D \phi \,dx  \to \int U^2 \bar A \D W \cdot \D \phi\,dx,
\]
for every $\phi \in C_c^\infty(\R^n)$. 

If we prove that the right-hand side of \eqref{eq:Wk:normalized} vanishes as $ k \to \infty $, then the limit function $ W $ solves the equation
\[
-\div(U^2\bar A \D W) = 0 \quad \text{in } \mathbb{R}^n.
\]
By Remark \ref{R:ratio}, it follows that the function $ V := U W $ is an entire solution to $ -\div(\bar A \D V) = 0 $, with $ Z(U) \subseteq Z(V) $. Finally, combining the growth condition \eqref{eq:growth:W} with the fact that $ \D W $ is not constant, we obtain a contradiction with the Liouville Theorem \ref{T:liouville}.

Hence, in order to conclude the proof, it remains to show that the right-hand side of \eqref{eq:Wk:normalized} vanishes as $ k \to \infty $. The purpose of the following section is to establish this fact.

\subsection{Estimates of the right-hand side}\label{S:4.1}

Our final goal is to show that the right-hand side of \eqref{eq:Wk:normalized} vanishes as $k\to\infty$. First, we have that
\begin{align*}
   & \frac{1}{ r_k^{\alpha} [\D w_k]_{C^{0,\alpha}(B_1)}}\int U_k^2 \bar A_k {\D w_k(\hat x_k)} \cdot \D \phi \,dx\\
    &= \frac{1}{ r_k^{\alpha} [\D w_k]_{C^{0,\alpha}(B_1)}} \int U_k^2(x) \big( A_k (\hat x_k+r_k x ) -A_k(\hat x_k)\big) \D w_k(\hat x_k )\cdot \D \phi(x)\,dx\\
    &+  \frac{1}{ r_k^{\alpha} [\D w_k]_{C^{0,\alpha}(B_1)}} \int U_k^2(x) (A_k \D w_k)(\hat x_k )\cdot \D \phi(x)\,dx\\
   & \le \|A_k\|_{ C^{0,1}(B_1)}r_k^{1-\alpha} \frac{\|\D w_k\|_{L^\infty(B_1)}}{[\D w_k]_{C^{0,\alpha}(B_1)}} +  \frac{1}{ r_k^{\alpha} [\D w_k]_{C^{0,\alpha}(B_1)}} \int U_k^2(x) (A_k \D w_k)(\hat x_k )\cdot \D \phi(x)\,dx\\
   &= o(1) - \frac{2}{ r_k^{\alpha} [\D w_k]_{C^{0,\alpha}(B_1)}} \int U_k(x)\D U_k(x)\cdot (A_k \D w_k)(\hat x_k) \phi(x) \,dx,
\end{align*}
where we used $\|A_k\|_{ C^{0,1}(B_1)}\le L$, $\|U_k\|_{L^\infty(B_R)}\le C_R$, $\|\D w\|_{L^\infty(B_{1})}\le k^{-1}[\D w_k]_{C^{0,\alpha}(B_{1/2})}$ and the divergence theorem. Moreover, using the orthogonality Assumption \ref{A:orthogonal}, we can expand this last expression as
\begin{align*}
   &\frac{1}{ r_k^{\alpha} [\D w_k]_{C^{0,\alpha}(B_1)}} \Big|\int U_k(x)\D U_k(x)\cdot (A_k \D w_k)(\hat x_k) \phi(x) \,dx\Big|\\
    &=  \Big| \sum_{i=1}^{n/2}  \frac{1}{ r_k^{\alpha} [\D w_k]_{C^{0,\alpha}(B_1)}} \int \frac{U_k^2(x)}{U_{i,k}^2(x)} U_{i,k}(x) \D U_{i,k}(x)\cdot (A_k \D w_k)(\hat x_k) \phi(x) \,dx\Big| \\
    & \le   \sum_{i=1}^{n/2} \Big\|\frac{U_k(x)}{U_{i,k}(x)}\Big\|_{L^\infty(B_R)}^2   \Big| \frac{1}{ r_k^{\alpha} [\D w_k]_{C^{0,\alpha}(B_1)}} \int  U_{i,k}(x) \D U_{i,k}(x)\cdot (A_k \D w_k)(\hat x_k) \phi(x) \,dx\Big|   =: \sum_{i=1}^{n/2} c_{i,k}^2  I_{i,k},
\end{align*}
where $c_{i,k} = \| U_k/U_{i,k}\|_{L^\infty(B_R)}\le C_R$ uniformly in $k$ and
\begin{equation*}
    I_{i,k}:=\Big| \frac{1}{ r_k^{\alpha} [\D w_k]_{C^{0,\alpha}(B_1)}} \int  U_{i,k}(x) \D U_{i,k}(x)\cdot (A_k \D w_k)( \hat x_k) \phi(x) \,dx\Big|
\end{equation*}
In order to conclude the proof and prove that the right-hand side of \eqref{eq:Wk:normalized} vanishes, we establish the following result:
\begin{Lemma}\label{L:I:ik} 
Under the previous notations, for every $i=1,\dots,n/2$, the term $I_{i,k} \to 0$ as $k \to \infty$.
\end{Lemma}

Let us recall that, using Assumption \ref{A:orthogonal}, we analyze each term $I_{i,k}$ separately, focusing solely on the two-dimensional variable $(x_{2i-1},x_{2i})$ and keeping all other variables fixed. With a slight abuse of notation, we introduce the conventions used throughout the remainder of the proof.

First, let us call $I_k:=I_{i,k}$.
We define $\zeta_k \in\R^n$ (which depends on $i$) to be a projection of the point $\hat x_k$ (which is defined in \eqref{eq:hat:x:ik}) onto the set ${R}(u_{i,k})$, where ${R}(u_{i,k})\subset \R^n$ is the regular part of the nodal set of $u_{i,k}$. Specifically, writing $\zeta_k=(\zeta_{1,k},\dots,\zeta_{n/2,k})$ with $\zeta_{j,k}\in\R^2$ for each $j$, we set
\[
\zeta_{j,k}:=\begin{cases}
    \hat{x}_{j,k} &\text{if } j\neq i,\\
    \xi_{i,k} & \text{if } j=i,
\end{cases}
\]
where $\xi_{i,k}$ is a chosen projection of $x_{i,k}$ onto the two-dimensional regular set $R(u_{i,k})$ restricted to $\R^2$. Note that $|\hat x_k-\zeta_k|=|\hat x_{i,k}-\xi_{i,k}|$. Let us also define 
\[
p_k^0:=\frac{\zeta_k-\hat x_k}{r_k}.
\]
Furthermore, we set
\[
 \delta_{k}:=\delta_{i,k}, \qquad u_k:=u_{i,k},\qquad U_k:=U_{i,k}.
\]
Finally, fixed $\e\in(0,1-\alpha)$, let us introduce the following quantity (which depends on the index $i$):
\[
\delta_k^*:=\inf\Big\{
\rho\in(0,1/8): N( \zeta_k,u_k,\rho)=1+\e
\big\} =  \inf\Big\{
\rho\in(0,1/8): N(p_k^0,U_k,\rho/r_k)=1+\e
\Big\},
\]
which plays the same role as the \emph{critical radius} in \cite[Definition 2.10]{TerTorVit24b}.

Heuristically, for every $2$-dimensional component $u_{i,k}$, the parameter $\delta_k$ measures the distance to the regular set, while $\delta_k^*$ characterizes the scale at which the singular set becomes visible in the blow-up limit. Consequently, three distinct geometric scenarios occur:
\begin{enumerate}
    \item[\emph{i)}] If $\delta_k \ge C > 0$, the blow-up sequence remains away from the zero set of $u_k$. In this case, $u_k$ behaves as a non-degenerate coefficient, and the analysis falls within the classical uniformly elliptic regularity theory.
    \item[\emph{ii)}] If $\delta_k \to 0$ and $\delta_k^* \ge C > 0$, the regular part $R(u_k)$ is captured at a microscopic scale, whereas the singular set remains far away and does not affect the blow-up procedure. Hence, we can adapt the regularity theory developed in \cite{TerTorVit24}.
    \item[\emph{iii)}] If $\delta_k \to 0$ and $\delta_k^* \to 0$, both the regular and the singular parts actively influence the blow-up limit. In this regime, to control the linear term $I_k$, we invoke Lemma \ref{L:2:points} to project $\nabla w_k(\hat x_k)$ onto $R(u_k)$. By exploiting the conormal boundary condition satisfied by $w_k$ (see \eqref{eq:conormal:wk}), we successfully achieve a quantitative control on the oscillation of $\nabla w_k(\hat x_k)$, which allows us to conclude the proof.
\end{enumerate}

\subsubsection{The uniformly elliptic case}

Let us suppose that $\delta_k \ge C>0$ uniformly in $k$, which means that the blow-up points stay away from the zero set of $u_k$. In this case, one has that $H(\hat x_k,u_k,r_k)^{1/2}\ge C >0$ and
\[|\D U_k(x)| = \frac{r_k |\D u_k(\hat x_k+r_k x)|}{H(\hat x_k,u_k,r_k)^{1/2}} \le C r_k,\]
hence, combining this estimates with \eqref{eq:1:alpha:contradiction}, it readily follows that
\[
I_k= \Big| \frac{1}{ r_k^{\alpha} [\D w_k]_{C^{0,\alpha}(B_1)}} \int  U_{k}(x) \D U_{k}(x)\cdot (A_k \D w_k)( \hat x_k) \phi(x) \,dx\Big| \le C r_k^{1-\alpha} \to 0.
\]

\subsubsection{The regular case} In this case we assume that $\delta_k \to 0$ and $\delta_k^*\ge \rho>0$, meaning that the blow-up points remain at a positive distance from the \emph{singular set} of $u_k$. We also emphasize that the proof for this case has been given in \cite[\S 2.5]{TerTorVit24b} and holds true in every dimension $n\ge2$. We include it here for completeness and because we adopt a slightly different and more concise approach. We distinguish two subcases.

\smallskip
 $\bullet$ Case $\delta_k/r_k\le C$.
\smallskip

In such a case, $\hat x_k= \zeta_k$, which implies that the centers of the blow-up sequence lie on the regular set $R(u_k)$.

For every $\beta \in (0,1)$ and $x\in \supp(\phi)$, by applying the doubling condition to $u_k$ (see Lemma \ref{L:doubling}), we obtain
\[
\begin{aligned}
|\D U_k(x)-\D U_k(0)| &= \frac{r_k| \D u_k(\zeta_k+r_k x)-\D u_k(\zeta_k)|}{H(\zeta_k,u_k,r_k)^{1/2}}
\\ 
&\le C r_k^{1+\beta} \frac{[\D u_k]_{C^{0,\beta}(B_{\rho/2})}}{H(\zeta_k,u_k,r_k)^{1/2}}
\le C r_k^{1+\beta} \frac{H(\zeta_k,u_k,\rho)^{1/2}}{H(\zeta_k,u_k,r_k)^{1/2}}
\le Cr_k^{\beta-\e}.
\end{aligned}
\]
Then, by invoking \eqref{eq:1:alpha:contradiction} and the conormal boundary condition $\D U_k(0)\cdot (A_k \D w_k)(\zeta_k)=0$ satisfied by $w_k$ on the regular set ${R}(u_k)$, we deduce
\[
\begin{aligned}
    I_k&=\frac{1}{ r_k^{\alpha} [\D w_k]_{C^{0,\alpha}(B_1)} }\Big|\int U_k(x) \D U_k(x) \cdot (A_k \D w_k)(\zeta_k) \phi(x) \,dx\Big|\\
    &\le \frac{1}{ r_k^{\alpha} [\D w_k]_{C^{0,\alpha}(B_1)} }\int  |U_k(x)| |\D U_k(x)-\D U_k(0)|  |(A_k \D w_k)(\zeta_k)| |\phi(x)|\,dx
    \le Cr_k^{\beta-\alpha-\e}\to 0,
\end{aligned}
\]
by choosing $\beta > \alpha+\e$.

\smallskip
 $\bullet$ Case $\delta_k/r_k\to\infty$.
\smallskip

In this case, the blow-up sequence remains at a distance of $\delta_k/r_k$ from the regular set $R(u_k)$.

We begin by recalling the following lemma, which holds true in every dimension $n\ge2$ and provides control over the oscillation of the normal vector to the regular set. 

\begin{Lemma}[{\cite[Lemma 2.12]{TerTorVit24b}}]
    Let $\mu>0$, $\alpha\in(0,1)$, $\e\in(0,1-\alpha)$, $\rho\in(0,1/8)$. Then, there exists $\bar r\in (0,\rho)$, depending on $\mu$, $\alpha$ and $\e$ such that the following holds true.  Let us suppose that $A\in\mathcal{A}$, $u\in H^1(B_1)$, $x_0\in {R}(u)\cap B_{3/4}$ such that
\[
\quad L_A u = 0 \text{ in }B_1, \quad \|u\|_{L^2(B_1)}=1,\quad N(x_0,u,r)\le 1+\e,\quad\text{for every }r\in(0,\rho).
\]
Then,    
\begin{equation}\label{eq:very:regular:gradient:oscillation}
       [\D u]_{C^{0,\alpha}(B_r(x_0))}\le \frac{\mu}{r^\alpha}|\D u(x_0)|, \quad \text{for every } r\in(0,\bar r).   
\end{equation}
\end{Lemma}

Rescaling the inequality \eqref{eq:very:regular:gradient:oscillation} to the function $U_k$, we obtain
\begin{equation}\label{eq:oscillation:rescaled}
    [\D U_k]_{C^{0,\alpha}(B_{\bar r/r_k}(p_k^0))} \le \mu \frac{r_k^\alpha}{\bar r^\alpha} |\D U_k(p_k^0)|.
    \end{equation}
Using $\D U_k(p_k^0)\cdot (A_k \D w_k)(\zeta_k)=0$ and \eqref{eq:oscillation:rescaled}, we deduce
\begin{align*}
I_k&=    \Big|\frac{1}{ r_k^{\alpha} [\D w_k]_{C^{0,\alpha}(B_1)} }\int_{B_R} U_k(x) \D U_k(x) \cdot (A_k \D w_k)(\hat x_k) \phi(x)\,dx \Big| \\
&=\Big|\frac{1}{ r_k^{\alpha} [\D w_k]_{C^{0,\alpha}(B_1)} }\int_{B_R} U_k(x) \D U_k(x) \cdot \Big( (A_k \D w_k)(\hat x_k)-(A_k \D w_k)(\zeta_k)+(A_k \D w_k)(\zeta_k) \phi(x)\Big) \,dx \Big|\\
&\le C \frac{1}{ r_k^{\alpha} [\D w_k]_{C^{0,\alpha}(B_1)} }\int_{B_R} |U_k(x)| | \D U_k(x) || (A_k \D w_k)(\hat x_k)-(A_k \D w_k)(\zeta_k)|\,dx\\
&+ C \frac{1}{ r_k^{\alpha} [\D w_k]_{C^{0,\alpha}(B_1)} }\int_{B_R} |U_k(x)| | (\D U_k(x) - \D U_k(p_k^0)) \cdot (A_k \D w_k)(\zeta_k)|\,dx\\
&\le C  \Big(\frac{\delta_k}{r_k}\Big)^\alpha \int_{B_R} |U_k(x)| | \D U_k(x) | \,dx + C r_k^{-\alpha} [\D U_k]_{C^{0,\alpha}(B_{\bar r/r_k}(p_k^0))}\Big(\frac{\delta_k}{r_k}\Big)^\alpha \int_{B_R}|U_k|\,dx\\
& \le C \Big(\frac{\delta_k}{r_k}\Big)^\alpha \|\D U_k\|_{L^\infty(B_R)} + C \Big(\frac{\delta_k}{r_k}\Big)^\alpha|\D U_k(p_k^0)|\\
&\le  C \Big(\frac{\delta_k}{r_k}\Big)^\alpha \|\D U_k\|_{L^\infty(B_{\delta_k/r_k})} .
\end{align*}
To show that $I_k \to 0$, we invoke Lemma \ref{L:Linfty:grad}. In fact, rescaling \eqref{eq:very:regular:bad:case2} in terms of $U_k$ yields
\begin{equation}\label{eq:L^infty:scaling:estimate:gradient}
    \|\D U_k \|_{L^{\infty}(B_{\delta_k/r_k})} \le C \frac{r_k}{\delta_k}|U_k(0)|\le C \frac{r_k}{\delta_k},
\end{equation}
from which the desired conclusion follows.

\subsubsection{The singular case}

Let us now assume that $\delta_k \to 0$ and $\delta_k^* \to 0$, meaning that the \emph{singular set} becomes \emph{visible} in the blow-up limit. We emphasize that this case is also treated in~\cite{TerTorVit24b}, where the authors prove analogous results using techniques and tools that are two-dimensional. Therefore, the arguments developed in the present section can be viewed as a generalization of the two-dimensional results to our setting.

By using Lemma \ref{L:2:points}, we obtain that there exists a new point $p_k^1 \in B_{R'}\cap {R}(U_k)$ such that 
\[
|\nu_k^0 \cdot \nu_k^1|\le 1-\e',\quad \text{ where }\, \nu_k^i:=\frac{\D U_k(p_k^i)}{|\D U_k(p_k^i)|}, \, \text{for } i=0,1.
\]
Let us define $\zeta_k^1 := r_k p_k^1 + \hat x_k$, which represents the point $p_k^1$ at the original scale, noting that this point remains close to $\zeta_k$ as $|\zeta_k^1 - \zeta_k| \leq C \delta_k^*$ by Lemma \ref{L:2:points}. We then introduce the unitary vector
\begin{equation}\label{eq:tau:k}
    \tau_k := \frac{\nu_k^1 - (\nu_k^0 \cdot \nu_k^1) \nu_k^0}{|\nu_k^1 - (\nu_k^0 \cdot \nu_k^1) \nu_k^0|},
\end{equation}
and observe that
\[
|\nu_k^1 - (\nu_k^0 \cdot \nu_k^1) \nu_k^0| \geq 1 - |\nu_k^0 \cdot \nu_k^1| \geq \varepsilon', \quad \text{and} \quad \tau_k \cdot \nu_k^0 = 0.
\]

\smallskip

  $\bullet$ Case $\delta_k/r_k \le C$ and $\delta_k^*/r_k \le C$.

\smallskip

In this case, we consider a normalized blow-up sequence $W_k$ which not contains the following linear terms:
\[
\D W_k(0)\cdot \nu^0_k,\quad \text{and} \quad \D W_k(0)\cdot \tau_k,
\]
noticing that $\text{span}(\nu_k,\tau_k) = \text{span}(e_{2i-1},e_{2i})$. Next, by using the conormal boundary condition \eqref{eq:conormal:wk}, which implies
\[
\D u_k(\zeta_k)\cdot (A_k \D w_k)(\zeta_k) = \D u_k(\zeta_k^1)\cdot (A_k \D w_k)(\zeta_k^1) = 0,
\]
we have that
\[
A_k(\zeta_k) \D W_k(0) \cdot \nu^0_k = 0,
\]
and
\begin{align*}
   & |A_k(\zeta_k) \D W_k(0)\cdot  \tau_k | = \Big|\frac{1}{ r_k^{\alpha} [\D w_k]_{C^{0,\alpha}(B_1)} } (A_k\D w_k)(\zeta_k) \cdot \tau_k\Big|  \\
   &=\Big|\frac{1}{ r_k^{\alpha} [\D w_k]_{C^{0,\alpha}(B_1)} } (A_k\D w_k)(\zeta_k) \cdot  \frac{\nu_k^1-\nu_k^0 (\nu_k^0 \cdot \nu_k^1)}{|\nu_k^1-\nu_k^0 (\nu_k^0 \cdot \nu_k^1)|}\Big|\\
   &=\Big|\frac{1}{ r_k^{\alpha} [\D w_k]_{C^{0,\alpha}(B_1)} } (A_k\D w_k)(\zeta_k) \cdot  \frac{\nu_k^1}{|\nu_k^1-\nu_k^0 (\nu_k^0 \cdot \nu_k^1)|}\Big|\\
   &\le \frac{|(A_k\D w_k)(\zeta_k) - (A_k\D w_k)(\zeta_k^1) |}{ r_k^{\alpha} [\D w_k]_{C^{0,\alpha}(B_1)} } \cdot  \Big|\frac{\nu_k^1}{|\nu_k^1-\nu_k^0 (\nu_k^0 \cdot \nu_k^1)|}\Big|
   \le \frac{L}{\e'} \Big(\frac{\delta_k^*}{r_k}\Big)^\alpha \le C.
\end{align*}
Hence, we have shown that the new sequence $W_k$ satisfies
\[
\lambda|\D W_k(0)|\le  |A_k(\zeta_k)\D W_k(0)|\le C.
\]
Then, the compactness argument involving the first-order expansion at zero of the sequence $ W_k $, as shown earlier, also applies to this new sequence, implying that $ W_k \to W $ in $ C^1_{\mathrm{loc}}(\mathbb{R}^n) $, for some possibly different limit $ W $ which still satisfies the properties established above. With this new choice, we obtain that $I_{i,k}=0$, since $W_k$ does not contain the linear part corresponding to the directions $e_{2i-1}$ and ${e_{2i}}$.

\smallskip
  $\bullet$ Case $\delta_k/r_k \le C$ and $\delta_k^*/r_k \to \infty$. 

\smallskip

By using the boundary condition \eqref{eq:conormal:wk} and recalling who is $\tau_k$ (see \eqref{eq:tau:k}), we deduce
\begin{align*}
I_k&=    \Big|\frac{1}{ r_k^{\alpha} [\D w_k]_{C^{0,\alpha}(B_1)} }\int_{B_R} U_k(x) \D U_k(x) \cdot (A_k \D w_k)(\zeta_k) \phi(x) \,dx\Big| \\
&
= 
\frac{1}{ r_k^{\alpha} [\D w_k]_{C^{0,\alpha}(B_1)} } \Big| \int_{B_R} U_k(x)  \partial_{\nu_k} U_k(x)  (A_k \D w_k)(\zeta_k) \cdot {\nu_k}\phi(x) \\
&\qquad\qquad\qquad\qquad\qquad+\int_{B_R} U_k(x)  \partial_{\tau_k} U_k(x)  (A_k \D w_k)(\zeta_k) \cdot {\tau_k} \phi(x)\,dx \Big|\\
& = \frac{1}{ r_k^{\alpha} [\D w_k]_{C^{0,\alpha}(B_1)} } \Big| \int_{B_R} U_k(x)  \partial_{\tau_k} U_k(x)  (A_k \D w_k)(\zeta_k) \cdot  \frac{\nu_k^1}{|\nu_k^1-\nu_k^0 (\nu_k^0 \cdot \nu_k^1)|}\phi(x) \,dx \Big|\\
& = \frac{1}{ r_k^{\alpha} [\D w_k]_{C^{0,\alpha}(B_1)} } \Big| \int_{B_R} U_k(x)  \partial_{\tau_k} U_k(x)  \Big((A_k \D w_k)(\zeta_k) - (A_k \D w_k)(\zeta_k^1)\Big) \cdot  \frac{\nu_k^1}{|\nu_k^1-\nu_k^0 (\nu_k^0 \cdot \nu_k^1)|} \phi(x)  \,dx\Big|\\
&\le \frac{C}{\e'}\Big(\frac{\delta_k^*}{r_k}\Big)^\alpha \int_{B_R} |U_k(x)||\partial_{\tau_k} U_k(x)| \,dx \le \frac{C}{\e'} \Big(\frac{\delta_k^*}{r_k}\Big)^\alpha [\partial_{\tau_k}U_k]_{C^{0,\beta}(B_{\delta_k^* /r_k})} \le \frac{C}{\e'} \Big(\frac{r_k}{\delta_k^*}\Big)^{\beta-\alpha-\e} \to 0, 
\end{align*}
where in the last inequality we used that $\partial_{\tau_k}U_k(0)=0$, since $\tau_k \cdot \nu_k = 0$ on the regular set ${R}(U_k)$ and the following scaling invariant Schauder estimates combined with the doubling property of $u_k$ (see Lemma \ref{L:doubling}),
\[
[\D U_k]_{C^{0,\beta}(B_{\delta_k^* /r_k})} \le C \Big(\frac{r_k}{\delta_k^*}\Big)^{1+\beta} \frac{H(\hat x_k,u_k,\delta_k^*)}{H(\hat x_k,u_k,r_k)} \le C \Big(\frac{r_k}{\delta_k^*}\Big)^{\beta-\e},
\]
choosing $\beta\in (0,1)$ such that $\beta-\alpha-\e>0$.

\smallskip
  $\bullet$ Case $\delta_k/r_k \to \infty$ and $\delta_k^*/\delta_k \le C$.

\smallskip

First, by using \eqref{eq:L^infty:scaling:estimate:gradient}, namely, $ \|\D U_k \|_{L^{\infty}(B_{\delta_k/r_k})} \le C {r_k}/{\delta_k},$  we get
\begin{align*}
I_k&=    \Big|\frac{1}{ r_k^{\alpha} [\D w_k]_{C^{0,\alpha}(B_1)} }\int_{B_R} U_k(x) \D U_k(x) \cdot (A_k \D w_k)(\hat x_k) \phi(x)\,dx  \Big| \\
& \le C \Big|\frac{1}{ r_k^{\alpha} [\D w_k]_{C^{0,\alpha}(B_1)} }\int_{B_R} U_k(x) |\D U_k(x)| \cdot |(A_k \D w_k)(\hat x_k)-(A_k \D w_k)(\zeta_k)|\,dx \Big|\\
&+C\Big|\frac{1}{ r_k^{\alpha} [\D w_k]_{C^{0,\alpha}(B_1)} }\int_{B_R} U_k(x) \D U_k(x) \cdot (A_k \D w_k)(\zeta_k) \,dx \Big|\\
&\le C \Big(\frac{\delta_k}{r_k}\Big)^\alpha \|\D U_k\|_{L^\infty(B_{\delta_k/r_k})}  +C\Big|\frac{1}{ r_k^{\alpha} [\D w_k]_{C^{0,\alpha}(B_1)} }\int_{B_R} U_k(x) \D U_k(x) \cdot (A_k \D w_k)(\zeta_k) \,dx \Big|\\
&\le C \Big(\frac{r_k}{\delta_k}\Big)^{1-\alpha}  +C\Big|\frac{1}{ r_k^{\alpha} [\D w_k]_{C^{0,\alpha}(B_1)} }\int_{B_R} U_k(x) \D U_k(x) \cdot (A_k \D w_k)(\zeta_k) \,dx \Big|.
\end{align*}
Let us define
\[
II_k:=\Big|\frac{1}{ r_k^{\alpha} [\D w_k]_{C^{0,\alpha}(B_1)} }\int_{B_R} U_k(x) \D U_k(x) \cdot (A_k \D w_k)(\zeta_k) \,dx \Big|.
\]
Using the same computations as in the previous case and \eqref{eq:L^infty:scaling:estimate:gradient}, we obtain
\begin{align*}
II_k&=    \Big|\frac{1}{ r_k^{\alpha} [\D w_k]_{C^{0,\alpha}(B_1)} }\int_{B_R} U_k(x) \D U_k(x) \cdot (A_k \D w_k)(\zeta_k)\,dx \Big| \\
&
= 
\frac{1}{ r_k^{\alpha} [\D w_k]_{C^{0,\alpha}(B_1)} } \Big| \int_{B_R} \Big( U_k(x)  \partial_{\nu_k} U_k(x)  (A_k \D w_k)(\zeta_k) \cdot {\nu_k} + U_k(x)  \partial_{\tau_k} U_k(x)  (A_k \D w_k)(\zeta_k) \cdot {\tau_k}\Big)\,dx \Big|\\
& = \frac{1}{ r_k^{\alpha} [\D w_k]_{C^{0,\alpha}(B_1)} } \Big| \int_{B_R} U_k(x)  \partial_{\tau_k} U_k(x)  (A_k \D w_k)(\zeta_k) \cdot  \frac{\nu_k^1}{|\nu_k^1-\nu_k^0 (\nu_k^0 \cdot \nu_k^1)|}\,dx \Big|\\
& = \frac{1}{ r_k^{\alpha} [\D w_k]_{C^{0,\alpha}(B_1)} } \Big| \int_{B_R} U_k(x)  \partial_{\tau_k} U_k(x)  \Big((A_k \D w_k)(\zeta_k) - (A_k \D w_k)(\zeta_k^1)\Big) \cdot  \frac{\nu_k^1}{|\nu_k^1-\nu_k^0 (\nu_k^0 \cdot \nu_k^1)|}\,dx  \Big|\\
&\le \frac{C}{\e'} \Big(\frac{\delta_k^*}{r_k}\Big)^\alpha \|\partial_{\tau_k}U_k\|_{L^{\infty}(B_{\delta_k /r_k})}\le 
\frac{C}{\e'} \Big(\frac{\delta_k^*}{r_k}\Big)^\alpha \Big(\frac{r_k}{\delta_k}\Big)
\le \frac{C}{\e'} \Big(\frac{r_k}{\delta_k}\Big)^{1-\alpha}\to 0, 
\end{align*}
and we conclude that $I_k\to 0$.

\smallskip
  $\bullet$ Case $\delta_k/r_k \to \infty$ and $\delta_k^*/\delta_k \to \infty$.  
 
\smallskip

First, we need the following scaling invariant estimates combined with the doubling property applied to $U_k$ (see Lemma \ref{L:doubling}) with balls $B_{\delta_k^*/r_k}$ and $B_{\delta_k/r_k}$
\begin{align*}
    [\D U_k]_{C^{0,\beta}(B_{\delta_k^*/r_k}(p_k^0))} &\le C \Big(\frac{r_k}{\delta_k^*}\Big)^{1+\beta} \Big(\fint_{B_{\delta_k^*/r_k}(p_k^0)}U_k^2 \,dx \Big)^{1/2} \le C \Big(\frac{r_k}{\delta_k^*}\Big)^{1+\beta} \Big(\frac{\delta_k^*}{\delta_k}\Big)^{1+\e} \Big(\fint_{B_{\delta_k/r_k}(p_k^0)}U_k^2 \,dx \Big)^{1/2}.
\end{align*}
Next, let us observe that
\begin{align*}
    \Big(\fint_{B_{\delta_k/r_k}(p_k^0)}U_k^2(x) \,dx  \Big)^{1/2} &= \Big(\fint_{B_{\delta_k/r_k}(p_k^0)}(U_k(x)-U_k(p_k^0))^2 \,dx \Big)^{1/2}\\
&\le C {\frac{\delta_k}{r_k}} [U_k]_{C^{0,1}(B_{\delta_k/r_k}(p_k^0))} 
\le C {\frac{\delta_k}{r_k}} [U_k]_{C^{0,1}(B_{2 \delta_k/r_k}(0))} \le C|U_k(0)|,
\end{align*}
where in the last inequality we argue exactly as in Lemma \ref{L:Linfty:grad}, remarking that $B_{\delta_k/r_k}(p_k^0)\subset B_{2 \delta_k/r_k}(0)$.

Hence, we proceed as in the previous case, but for the last line in the estimate of $II_k$ we obtain 
\begin{align*}
     II_k &\le C \Big(\frac{\delta_k^*}{r_k}\Big)^\alpha \int_{B_R} |\partial_{\tau_k} U_k(x)-\partial_{\tau_k}U_k(p_k^0)|\,dx\le C \Big(\frac{\delta_k^*}{r_k}\Big)^\alpha \Big(\frac{\delta_k}{r_k}\Big)^\beta [\partial_{\tau_k}U_k]_{C^{0,\beta}(B_{\delta_k^*/r_k}(p_k^0))}\\
    &\le C \Big(\frac{r_k}{\delta_k^*}\Big)^{1+\beta-\alpha} \Big(\frac{\delta_k}{r_k}\Big)^\beta  \Big(\frac{\delta_k^*}{\delta_k}\Big)^{1+\e} \Big(\fint_{B_{\delta_k/r_k}(p_k^0)}U_k^2 \,dx \Big)^{1/2} \le C \Big(\frac{r_k}{\delta_k}\Big)^{1-\alpha} \Big(\frac{\delta_k}{\delta_k^*}\Big)^{\beta-\e-\alpha}\to 0,
\end{align*}
since $\beta-\e-\alpha>0$. Thus, in all cases, we have shown that $I_k$ vanishes as $k \to \infty$, which concludes the proof of Lemma \ref{L:I:ik}. By repeating this argument for each term $I_{i,k}$, we deduce that the right-hand side of \eqref{eq:Wk:normalized} vanishes as $k \to \infty$, completing the proof.
\end{proof}

\section*{Acknowledgment} I would like to thank Susanna Terracini for many fruitful discussions on the topic.
The author is supported by the INdAM - GNAMPA Project "Struttura fine e regolarit\`a in problemi variazionali non-lineari" codice CUP E53C25002010001.

\end{document}